\documentclass[a4paper,11pt]{amsart}
\usepackage[utf8]{inputenc}
\usepackage{amsmath}
\usepackage{amsfonts}
\usepackage{amsthm}
\usepackage{amssymb}
\usepackage{color}
\usepackage{tikz}
\usepackage[margin=1in]{geometry}
\usepackage{todonotes}

\newcommand{\set}[1]{\left\{ #1 \right\}}
\newcommand{\norm}[1] {\left\| #1 \right\|}

\newcommand{\path}[3]{\mbox{path}_{#1}{(#2,#3)}}
\newcommand{\wmu}{\widetilde{\mu}}
\newcommand{\mbf}[1]{\mathbf{#1}}
\newcommand{\supp}[1]{\mathrm{Supp}(#1)}

\parskip\bigskipamount

\theoremstyle{plain}
\newtheorem{thm}{Theorem}[section]
\newtheorem{lem}[thm]{Lemma}
\newtheorem{prop}[thm]{Proposition}

\theoremstyle{definition}
\newtheorem{defn}[thm]{Definition}
\newtheorem{exmp}[thm]{Example}

\theoremstyle{remark}
\newtheorem{rem}[thm]{Remark}
\newcommand{\R}{\mathbb{R}}
\begin{document}
\title{Embeddability of generalized wreath products and box spaces}
\address{School of Mathematics,
University of Southampton, Highfield, Southampton, SO17 1BJ, United Kingdom.}
\author{Chris Cave}
\email[Chris Cave]{Chris.Cave89@gmail.com}

\address{School of Mathematics,
University of Southampton, Highfield, Southampton, SO17 1BJ, United Kingdom.}
\author{Dennis Dreesen}
\email[Dennis Dreesen]{Dennis.Dreesen@soton.ac.uk}

\address{(1) Clermont Université, Université Blaise Pascal, Laboratoire de Mathématiques, BP 10448, F-63000 Clermont-Ferrand, FRANCE
(2) CNRS, UMR 6620, LM, F-63171 AUBIERE, FRANCE
}
\author{Ana Khukhro}
\email[Ana Khukhro]{Ana.Khukhro@math.univ-bpclermont.fr}
\renewcommand{\thefootnote}{\fnsymbol{footnote}} 
\footnotetext{The first author is sponsored by the EPSRC. The second author is a Marie Curie Intra-European Fellow within the 7th European Community Framework Programme.}
\renewcommand{\thefootnote}{\arabic{footnote}} 

\begin{abstract}
Given two finitely generated groups that coarsely embed into a Hilbert space, it is known that their wreath product also embeds coarsely into a Hilbert space. We introduce a wreath product construction for general metric spaces $X,Y,Z$ and derive a condition, called the ($\delta$-polynomial) path lifting property, such that coarse embeddability of $X,Y$ and $Z$ implies coarse embeddability of $X\wr_Z Y$. We also give bounds on the compression of $X\wr_Z Y$ in terms of $\delta$ and the compressions of $X,Y$ and $Z$. Next, we investigate the stability of the property of admitting a box space which coarsely embeds into a Hilbert space under the taking of wreath products. We show that if an infinite finitely generated residually finite group $H$ has a coarsely embeddable box space, then $G\wr H$ has a coarsely embeddable box space if $G$ is finitely generated abelian. This leads, in particular, to new examples of bounded geometry coarsely embeddable metric spaces without property $A$.
\end{abstract}
\maketitle
\section{Introduction}
Ever since the recently discovered relations between the Novikov conjecture and coarse embeddability \cite{STY}, this latter property has been the focal point of much research.
Concretely, for a finitely generated group with a word metric relative to a finite generating subset, coarse embeddability into a Hilbert space implies the Novikov conjecture. This result was suggested by Gromov in \cite[Problems (4) and (5)]{Ferry1995} and proven in \cite{STY}. Later, in \cite{KY}, the same result was proved for embeddings into uniformly convex Banach spaces, providing one of the motivations for studying embeddings into $l^p$-spaces for $p\neq 2$.
\begin{defn}[see \cite{G01}]
Fix $p\geq 1$. A metric space $(X,d)$ is {\em coarsely embeddable into an $L^p$-space} if there exists a measure space $(\Omega,\mu)$, non-decreasing functions $\rho_-,\rho_+:\R^+ \rightarrow \R^+$ such that $\lim_{t\to \infty} \rho_-(t)=+\infty$ and a map $f:X\rightarrow L^p(\Omega,\mu)$, such that
\[ \rho_-(d(x,x')) \leq \|f(x)-f(x')\|_p \leq \rho_+(d(x,x')) \ \forall x,x'\in X. \]
The map $f$ is called a {\em coarse embedding} of $X$ into $L^p(\Omega,\mu)$ and the map $\rho_-$ is called a {\em compression function} for $f$.
\end{defn}
In $2004$, Guentner and Kaminker introduced a numerical invariant that can be used to quantify ``how well'' a metric space $(X,d)$ embeds coarsely into a Hilbert space \cite{GK04}. This links coarse embeddability to the well-studied notion of quasi-isometric embeddability \cite{dCTV07}.
\begin{defn}
Fix $p\geq 1$. Given a metric space $(X,d)$ and a measure space $(\Omega,\mu)$, the {\em $L^p$-compression} $R(f)$ of a coarse embedding $f:X\rightarrow L^p(\Omega,\mu)$ is defined as the supremum of $r\in [0,1]$ such that
\[ \exists C,D>0, \forall x,x'\in X: \frac{1}{C} d(x,x')^r - D \leq \| f(x)-f(x')\| \leq C d(x,x') + D.\]
If such $r$ does not exist, then we set $R(f)=0$. The {\em $L^p$-compression} $\alpha_p(X)$ of $X$ is defined as the supremum of $R(f)$ taken over all coarse embeddings of $X$ into all possible $L^p$-spaces.
\end{defn}
In the setting of groups, compression is related to interesting group-theoretic properties. For example, it is known that finitely generated groups with non-equivariant compression $>1/2$ satisfy property $A$ (which is equivalent to exactness of the reduced $C^*$-algebra) \cite{GK04}. The converse is not true. Other interesting facts occur in the amenable case. If $G$ is an amenable group, then given a coarse embedding $f:G\rightarrow \mathcal{H}$ with compression $R(f)$, one can always find an affine isometric action of the group on a Hilbert space such that the associated $1$-cocycle also has compression $R(f)$ \cite{dCTV07}. This is related to properties such as Kazhdan's property $(T)$ and the Haagerup property.

\begin{defn}
Let $G$ be a group and $(\Omega,\mu)$ a measure space. Fix $p\geq 1$. A map $f:G\rightarrow L^p(\Omega,\mu)$ is called {\em $G$-equivariant} if there is an affine isometric action $\alpha$ of $G$ on $L^p(\Omega,\mu)$ such that $\forall g,h\in G: f(gh)=\alpha(g)(f(h))$.
A compactly generated, locally compact, second countable group $G$ equipped with the word length metric relative to a compact generating subset is said to satisfy the {\em Haagerup property} if it admits an equivariant coarse embedding into a Hilbert space.
\end{defn}

A lot of effort has gone into studying the behaviour of coarse embeddability and the Haagerup property under group constructions 
\cite{DG08}, \cite{CCJJV01}. In \cite{Li2010}, S. Li gave a proof of the fact that the wreath product of two countable groups with the Haagerup property is again Haagerup. Although using similar ideas, his proof is more concise than that of \cite{CSV12}, where the authors prove a more general statement. Instead of looking only at standard wreath products $G\wr H$, Cornulier, Stalder and Valette considered permutational wreath products $G\wr_X H:=G^{(X)}\rtimes H$, where $X$ is a countable $H$-set and $H$ acts on $G^{(X)}$ by shifting indices. They conjectured that the Haagerup property for $G$ and $H$ would imply the Haagerup property for any permutational wreath product $G\wr_X H$, but only proved it in the case where $X=H/L$ with $L$ {\em co-Haagerup} in $H$. Here, a subgroup $L<H$ is called {\em co-Haagerup} if there exists a proper $G$-invariant conditionally negative definite kernel on $H/L$. It was shown in \cite{Ioana} that the above mentioned conjecture is false and so the choice of $X$ is restricted.

The non-equivariant analogue of the Haagerup property is {\em coarse embeddability}. By the work of Dadarlat and Guentner, it follows that $G\wr H$ is coarsely embeddable if $G$ is coarsely embeddable and $H$ has property $A$. It is known that property $A$ implies coarse embeddability into a Hilbert space but the converse is unknown in the case of finitely generated groups. Li in \cite{Li2010} and Cornulier, Stalder and Valette in \cite{CSV12} show that coarse embeddability into a Hilbert space is preserved under wreath products, without referring to property $A$. Even stronger, for $p\in [1,2]$, Li in \cite{Li2010} proved that the $L^p$-compression of a wreath product $G\wr H$ is strictly positive whenever $G$ and $H$ have strictly positive $L^p$-compression. Although coarse embeddability and compression are defined for arbitrary metric spaces, the behaviour of compression and coarse embeddability had not been studied for any type of permutational wreath products.

Yu originally defined property $A$ as a means of guaranteeing the existence of a coarse embedding into a Hilbert space \cite{Y00}. The question of whether the two properties are actually equivalent went unanswered until Nowak gave an example of a metric space with a coarse embedding into a Hilbert space which does not have property $A$. Nowak's space, a disjoint union of $n$-dimensional cubes $\{0,1\}^n$ over all $n\in \mathbb{N}$, is not of bounded geometry. In \cite{Arzhantseva2012}, Arzhantseva, Guentner and Spakula provided the first bounded geometry example of a metric space which coarsely embeds into a Hilbert space but does not have property $A$ in the form of a \emph{box space} of a finitely generated free group (see definition \ref{def:boxspace}). The fact that it does not have property A follows from the work of Guentner (see Proposition 11.39 in \cite{R03}), where it is shown that any (or equivalently all) box spaces of a finitely generated residually finite group have property $A$ if and only if the group is amenable.
\newpage
In this paper, we define a general permutational wreath product $X\wr^C_Z Y$ of arbitrary metric spaces $X,Y,Z$ where $C\in \R^+$ and we investigate under which conditions the coarse embeddability of $X,Y,Z$ implies coarse embeddability of $X\wr_Z Y$. This leads to the definition of the ($\delta$-polynomial) path lifting property. Precisely, we obtain the following result. Our proof uses similar ideas as in \cite{CSV12} and \cite{Li2010}.
 \begin{thm}[see Theorem \ref{CEresult}]
  Let $X,Y,Z$ be metric spaces and $p \colon Y \to Z$ be a $C$-dense bornologous map with the coarse path lifting property. Assume that $Y$ is uniformly discrete and that $Z$ has $C$-bounded geometry. If $X,Y,Z$ are coarsely embeddable into a Hilbert space, then so is $X \wr_Z^C Y$.
 \end{thm}
 We also give bounds on the Hilbert space compression of this wreath product in terms of $\delta$ and the Hilbert space compression of $X,Y$ and $Z$. The bounds we obtain coincide with the bounds given in \cite{Li2010} when applied to standard wreath products.

%


Next, we turn to box spaces. We remark that by a result of Gruenberg \cite{Gru}, the wreath product $G\wr H$ of two residually finite groups $G$ and $H$ is itself residually finite if and only if either $H$ is finite or $G$ is abelian.
We prove the following result, obtaining new examples of coarsely embeddable bounded geometry metric spaces without property $A$.
\begin{thm}[see Theorem \ref{BoxResult}]\label{thm:intro}
Let $G$ be a finitely generated abelian group and let $H$ be a finitely generated residually finite group which has a box space which embeds coarsely into a Hilbert space. Then there is a box space of the wreath product $G\wr H$ which coarsely embeds into a Hilbert space.
\end{thm} 

The above result should also be considered in the following setting. Let $G$ be a finitely generated residually finite group, and $\{K_i\}$ a collection of finite index subgroups with trivial intersection. Roe shows in \cite{R03} that if the box space $\Box_{\{K_i\}} G$ is coarsely embeddable into a Hilbert space, then $G$ has the Haagerup property. 
This implication is not reversible. In fact, there exist groups with the Haagerup property for which every box space is an expander. For finitely generated residually finite groups, we can thus think of the property of admitting a box space which coarsely embeds into a Hilbert space as lying strictly between amenability and the Haagerup property. One can begin investigating which groups lie in this class by proving stability results for various group constructions. In \cite{Khu}, this is done for certain group extensions and our Theorem \ref{thm:intro} treats the case of wreath products.

\section{Preliminaries}
Given two finitely generated groups $G$ and $H$, the wreath product, written as $G \wr H$ is the set of pairs $(\mathbf{f},h)$ where $h \in H$ and $\mathbf{f} \colon H \to G$ is a finitely supported function (i.e. $\mathbf{f}(h) = e_G$ for all but finitely many $h \in H$) together with a group operation
\begin{equation*}
 (\mathbf{f},h) \cdot (\mathbf{g}, h') = (\mathbf{f} \cdot (h \mathbf{g}), hh') 
\end{equation*}
where $(h \mathbf{g})(z) = \mathbf{g}(h^{-1} z)$ for all $z \in G$. One can think of $G \wr H$ as being the semi-direct product $\bigoplus_{H}G \rtimes H$ where $H$ acts on $\bigoplus_{H}G$ by permuting the indices. If finite sets $S$ and $T$ generate $G$ and $H$ respectively then $G \wr H$ is generated by the finite set
\begin{equation*}
 \set{(\mathbf{e},t) : t \in T} \cup \set{(\delta_s, e_H) : s \in S}
\end{equation*}
where $\mathbf{e}(h) = e_G$ for all $h \in H$ and \begin{equation*}\delta_s(h) = \begin{cases}
                                                                  s & \quad \mbox{if $h = e_H$}\\
                                                                  e_G & \quad \mbox{otherwise.}
                                                                 \end{cases}
                                                                 \end{equation*}
                                                                 
The word metric on $G \wr H$ coming from this generating set can be thought of as follows. Given two elements $(\mathbf{f},x)$ and $(\mathbf{g},y)$, take the shortest path in the Cayley graph Cay$(H,T)$ going from $x$ to $y$ that passes through the points in $\mbox{Supp}(\mathbf{f}^{-1} \mathbf{g}) = \set{h_1, \ldots, h_n}$. At each point  $h_i \in \mbox{Supp}(\mathbf{f}^{-1} \mathbf{g})$ travel from $\mbf{f}(h_i)$ to $\mbf{g}(h_i)$ in $G$. Explicitly for $(\mathbf{f},x), (\mathbf{g},y) \in \bigoplus_{g \in H} G \rtimes H$ and $\mbox{Supp }(\mathbf{f}^{-1} \mathbf{g}) = \set{h_1, \ldots, h_n}$ define
\begin{equation*}
 p_{(x,y)}(\mathbf{f},\mathbf{g}) = \inf_{\sigma \in S_n} \left( d_H(x, h_{\sigma(1)}) + \sum_{i=1}^{n} d_H(h_{\sigma(i)}, h_{\sigma(i+1)}) + d_H(h_{\sigma(n)}, y) \right).
\end{equation*}
where the infimum is taken over all permutations in $S_n$. The number $p_{(x,y)}(\mathbf{f},\mathbf{g})$ corresponds to the shortest path between $x$ and $y$ in $H$ going through each element in $\supp{\mathbf{f}^{-1} \mathbf{g}}$. Hence the distance between $(\mathbf{f},x)$ and $(\mathbf{g},y)$ is
\begin{equation*}
 d_{G \wr H} ((\mathbf{f},x),(\mathbf{g},y) ) = p_{(x,y)}(\mathbf{f}, \mathbf{g}) + \sum_{h \in H} d_G(\mathbf{f}(h), \mathbf{g}(h))
\end{equation*}
Suppose $G$, $H$ are groups and $H$ acts transitively on a set $X$. Fix a base point $x_0 \in X$ and define the \emph{permutational wreath product} to be the group $G \wr_X H \mathrel{\mathop:}=  \bigoplus_{X} G \rtimes H$ where \begin{equation*}
\bigoplus_X G = \set{\mbf{f} \colon X \to G : \mbf{f}(x) = e_G \mbox{ for all but finitely many $x \in X$} }                                                                                                                                                                                                                                     \end{equation*}
and $H$ acts on $\bigoplus_{X} G$ by permuting the indices. If $S$ and $T$ generate $G$ and $H$ respectively then $G \wr_X H$ is generated by 
\begin{equation*}
 \set{(\mathbf{e},t) : t \in T} \cup \set{(\delta_s, e_H) : s \in S}
\end{equation*}
where $\mathbf{e}(x) = e_G$ for all $x \in X$ and \begin{equation*}
                                                   \delta_s(x) = \begin{cases}
                                                               s & \quad \mbox{if $x = x_0$}\\
                                                               e_G &\quad \mbox{otherwise}.                                                             
                                                              \end{cases}
                                                  \end{equation*}
The metric on $G \wr_X H$ from the generating set can be thought of as follows. Given two elements $(\mathbf{f},x)$ and $(\mathbf{g},y)$ take the shortest path going from $x$ to $y$ in Cay$(H,T)$ that passes through points $\set{h_1, \ldots, h_n}$ such that $\mbox{Supp}(\mathbf{f}^{-1} \mathbf{g}) = \set{h_1x_0, \ldots, h_nx_0}$. At each element  $h_i \in \mbox{Supp}(\mathbf{f}^{-1} \mathbf{g})$ travel from $\mbf{f}(h_ix_0)$ to $\mbf{g}(h_ix_0)$ in $G$. In general the shortest path is not necessarily unique.

Explicitly for $(\mathbf{f},x), (\mathbf{g},y) \in \bigoplus_{x \in X} G \rtimes H$, let $I = \supp{\mbf{f}^{-1}\mbf{g}}$ and let $n = |\supp{\mbf{f}^{-1}\mbf{g}}|$. Define $\mathcal{P}_I$ to be the set
\begin{equation*}
 \mathcal{P}_I \mathrel{\mathop:}= \set{(h_1, \ldots, h_n) \subset H^n : \set{h_1x_0, \ldots, h_n x_0} = I}.
\end{equation*}
In particular if $(h_1, \ldots, h_n) \in \mathcal{P}_I$ then any permutation of $(h_1, \ldots, h_n)$ is also in $\mathcal{P}_I$. Hence the length of the shortest path between $x$ and $y$ in $H$ passing though the points that project onto $\supp{\mbf{f}^{-1} \mbf{g}}$ is precisely
\begin{equation*}
 \rho_{(x,y)}(\mathbf{f}, \mathbf{g}) \mathrel{\mathop:}= \inf_{(h_1, \ldots, h_n) \in \mathcal{P}_I} \left( d(x, h_1) + \sum_{i=1}^{n-1} d(h_i, h_{i+1}) + d(h_n, y) \right).
\end{equation*}
Hence the distance between $(\mathbf{f},x)$ and $(\mathbf{g},y)$ is
\begin{equation*}
 d_{G \wr_X H} ((\mathbf{f},x),(\mathbf{g},y)) = \rho_{(x,y)}(\mathbf{f}, \mathbf{g}) + \sum_{z \in X} d_G(\mathbf{f}(z), \mathbf{g}(z)).
\end{equation*}
One can ask whether we can generalise this construction. 
Suppose $Y$ and $Z$ are metric spaces and $p \colon Y \to Z$ is a $C$-dense map, i.e. $B_Z(p(Y),C)=Z$. Given two points $y, y' \in Y$ and a finite sequence of points $I = \set{z_1, \ldots, z_n}$ in $Z$, we define $\mathcal{P}_I$ to be the set 
\begin{equation*}
 \mathcal{P}_I \mathrel{\mathop:}= \set{(y_1, \ldots, y_n) \subset Y^n : \exists \sigma \in S_n \mbox{ such that } \forall i, \ p(y_i) \in B(z_{\sigma(i)}, C)}.
 \end{equation*}
 In particular, if $(y_1, \ldots, y_n) \in \mathcal{P}_I$ then any permutation of $(y_1, \ldots, y_n)$ also lies in $\mathcal{P}_I$. We now define the \emph{length of the path from $y$ to $y'$ going through $I$} by
\begin{equation*}
 \path{I}{y}{y'} = \inf_{(y_1, \ldots, y_n) \in \mathcal{P}_I} \left( d_Y(y, y_1) + \sum_{i=1}^{n-1} d_Y(y_i, y_{i+1}) + d_Y(y_n, y') \right).
\end{equation*}
Let $X$ be another metric space and fix a base point $x_0 \in X$. Define $\bigoplus_Z X$ to be the set
\begin{equation*}
\bigoplus_Z X = \set{\mbf{f} \colon Z \to X : \mbf{f}(z) = x_0 \mbox{ for all but finitely many $z \in Z$}}.
\end{equation*}
For $\mbf{f}, \mbf{g} \in \bigoplus_Z X$ define $\supp{\mbf{f}^{-1} \mbf{g}} = (\supp{\mbf{f}} \cup \supp{\mbf{g}}) \setminus \set{z \in Z : \mbf{f}(z) = \mbf{g}(z)}$. Let $(\mbf{f},y),(\mbf{g},y') \in \bigoplus_Z X \times Y$ and let $I = \supp{\mbf{f}^{-1} \mbf{g}}$. Define a metric on the set $\bigoplus_Z X \times Y$ by
\begin{equation*}
 d((\mbf{f},y),(\mbf{g},y')) = \path{I}{y}{y'} + \sum_{z \in Z}d_X(\mbf{f}(z),\mbf{g}(z)).
\end{equation*}
We obtain a metric space $(\bigoplus_Z X \times Y,d)$, which we denote by $X \wr_Z^C Y$. When there is no risk for confusion, we will omit $C$ from this notation. When $X$ and $Y$ are graphs, then the metric wreath product $X \wr_Y Y$ coincides with the wreath product of graphs. See Definition 2.1 in \cite{Erschler2006}.

 \begin{defn}
Given two metric spaces $(X,d_X)$ and $(Y,d_Y)$, a \emph{coarse embedding} is a map $f \colon X \to Y$ such that there exist non-decreasing maps $\rho_-, \rho_+ \colon \mathbb{R}^+ \to \mathbb{R}^+$ where $\rho_{-,+}(t) \to \infty$ as $t \to \infty$ and
\begin{equation*}
 \rho_-(d_X(x,y)) \leq d_Y(f(x),f(y)) \leq \rho_+(d_X(x,y))
\end{equation*}
for all $x,y \in X$. A metric space $(X,d_X)$ is \emph{coarsely embeddable in a Hilbert space} if there exists a map that is a coarse embedding and its codomain is a Hilbert space. We usually shorten this property to saying $X$ is CE.
\end{defn}


\section{Measured walls \label{section:wallspaces}}
Let $X$ be a set and $2^X$ the power set of $X$. We endow $2^X$ with the product topology. For $x \in X$, denote $\mathcal{A}_x = \set{A \subset X : x \in A}$. This is a clopen subset in $2^X$. For two elements $x,y \in X$ we say a set $A \subset X$ \textit{cuts} $x$ and $y$, denoted $A \vdash \set{x,y}$ if $x \in A$ and $y \in A^c$ or $x \in A^c$ and $y \in A$. Likewise we say that $A$ cuts another set $Y$ if neither $Y \subset A$ nor $Y \subset A^c$.
\begin{defn}
 A \textit{measured walls structure} on a set $X$ is a Borel measure $\mu$ on $2^X$ such that for every $x,y \in X$,
 \begin{equation*}
  d_{\mu} (x,y) \mathrel{\mathop:}= \mu \left( \set{A \in 2^X : A \vdash \set{x,y} }\right) < \infty.
 \end{equation*}
\end{defn}
Since $\set{A \in 2^X : A \vdash \set{x,y}} = \mathcal{A}_x \bigtriangleup \mathcal{A}_y$, the set is measurable. It follows that $d_{\mu}$ is well defined and is a pseudometric on $X$, called the \textit{wall metric associated to $\mu$}.

\begin{exmp}
 Let $X$ be a set. A \emph{wall} is a set of $2^X$ of the form $\set{A, A^c}$ for some $A \subset X$. A \emph{space with walls} is a pair $(X, \mathcal{W})$ where $\mathcal{W}$ is collection of walls and for each $x,y \in X$ the number of walls $w(x,y)$ separating $x$ from $y$ is finite. That is
 \begin{equation*}
  w(x,y) = |\set{\set{A,A^c} \in \mathcal{W}: A \vdash \set{x,y}}| < \infty
 \end{equation*}
 For a collection of walls $\mathcal{W}$, define $\mathcal{H} \subset 2^X$ to be the set $\set{A \subset X : \set{A,A^c} \in \mathcal{W}}$. We call $\mathcal{H}$ the set of \emph{half-spaces} of $\mathcal{W}$ and $A\subset X$ a \emph{half-space} if $A \in \mathcal{H}$. For $B \subset 2^X$ define 
 \begin{equation*}
\mu(B) = \frac{1}{2} \sum_{A \in B \cap \mathcal{H}} 1  
 \end{equation*}
This is a measured walls structure and the associated pseudometric $d_{\mu}$ is the number of walls separating two points.
\end{exmp}
If $f \colon X \to Y$ is a map between sets and $(Y, \mu)$ is a measured walls structure, then we can push forward the measure $\mu$ by the inverse image map $f^{-1} \colon 2^Y \to 2^X$ and obtain a measure walls structure $(X, f^* \mu)$, where for $A \subset 2^X$, $f^* \mu (A) = \mu(\{f(B)\mid B\in A, B=f^{-1}(f(B))\})$. It follows that $d_{f^* \mu}(x,x') = d_{\mu}(f(x), f(x'))$.

Given a family of spaces $X_i$ with measured walls space structures $\mu_i$ and the natural projection maps $p_i \colon \bigoplus X_j \to X_i$, then the measure $\mu = \sum_I p_i^* \mu_i$ defines a measured walls space structure on $\bigoplus_i X_i$. The associated wall metric is $d_{\mu}((x_i), (y_i)) = \sum_i d_{\mu_i}(x_i, y_i)$.

\begin{defn}
Let $X$ be a set. A function $k \colon X \times X \to \mathbb{R}_+$ is a \emph{conditionally negative definite kernel} if $k(x,x) = 0$ and $k(x,y) = k(y,x)$ for all $x,y \in X$ and for every finite sequences $x_1, \ldots , x_n \in X$, $\lambda_1, \ldots, \lambda_n$ such that $\sum_{i=1}^n \lambda_i = 0$ we have
\begin{equation*}
 \sum_{i,j}\lambda_i \lambda_j k(x_i, x_j) \leq 0
\end{equation*}
\end{defn}

\begin{prop}[Proposition 6.16 in \cite{Chatterji2010}, see also Proposition 2.6 in \cite{CSV12}] \label{propembed}
 Let $X$ be a set and $k \colon X \times X \to \mathbb{R}_{+}$. The following are equivalent:
 \begin{enumerate}
  \item there exists $f \colon X \to L^1(X)$ such that $k(x,y) = \norm{f(x) - f(y)}_1$ for all $x,y \in X$;
  \item For every $p\geq 1$, there exists $f \colon X \to L^p(X)$ such that $(k(x,y))^{1/p} = \norm{f(x) - f(y)}_p$ for all $x,y \in X$;
  \item $k = d_{\mu}$ for some measured walls structure $(X, \mu)$.
 \end{enumerate}
\end{prop}
In order to prove our main result we make use of a method of lifting measured walls structures. First we require some technical definitions. Let $W,X$ be sets and $\mathcal{A} = 2^{(X)}$, the set of finite subsets of $X$.

\begin{defn}
 An $\mathcal{A}$-gauge on $W$ is a function $\phi \colon W \times W \to \mathcal{A}$ such that:
 \begin{align*}
  \phi(w,w') & = \phi(w',w) \quad \forall w,w' \in W \\
  \phi(w,w'') & \subset \phi(w,w') \cup \phi(w',w'') \quad \forall w,w',w'' \in W
 \end{align*}
\end{defn}
If $W$ is a group then $\phi$ is called \emph{left invariant} if $\phi(w w',w w'') = \phi(w', w'')$ for all $w,w',w'' \in W$.

\begin{defn}
 Let $G$ be a group and $X$ a $G$-set. A measured walls structure $(X,\mu)$ is \emph{uniform} if for all $x,y \in X$ the map $g \mapsto d_{\mu}(gx,gy)$ is bounded on $G$.
\end{defn}

\begin{thm}
 Let $X,W$ be sets, $\mathcal{A} = 2^{(X)}$. Let $\phi$ be an $\mathcal{A}$-gauge on $W$ and assume that $\phi(w,w) = \emptyset$ for all $w \in W$. Let $(X, \mu)$ be a measured walls structure.
 \begin{enumerate}
  \item [(i)] There is a naturally defined measure $\widetilde{\mu}$ on $2^{W \times X}$ such that $(W \times X, \widetilde{\mu})$ is a measured walls structure with corresponding pseudometric
  \begin{equation*}
   d_{\widetilde{\mu}}(w_1 x_1, w_2 x_2) = \mu \left( \set{A \in \mathcal{A} : A \vdash \phi(w_1, w_2) \cup \set{x_1, x_2}} \right).
  \end{equation*}
  \item [(ii)] Let $H$ be a group and suppose that $X$ is an $H$-set and $W$ is an $H$-group where the action is by automorphisms. Suppose $\phi$ is $W$-invariant and $H$-equivariant. If $(X,\mu)$ is invariant (uniform) under $H$, then $(W \times X, \widetilde{\mu})$ is invariant (uniform) under $W \rtimes H$
 \end{enumerate}
\label{thm:wallsandembeddings}
\end{thm}
The proof of Theorem \ref{thm:wallsandembeddings} can be found in \cite{CSV12}. A consequence of this theorem is that if $X,Y,Z$ are metric spaces where $X$ has a fixed point $x_0 \in X$ then $\supp{\mbf{f}^{-1} \mbf{g}}$ is an $\mathcal{A}$-gauge on $\bigoplus_Z X$, where $\mathcal{A} = 2^{(Z)}$. Hence if $Z$ has a measured walls structure there exists a lifted measured walls structure on $\bigoplus_Z X \times Z$.

\section{Coarse embeddings of wreath products}
\begin{defn}
A metric space $(X,d)$ is \emph{uniformly discrete} if there exists $\delta > 0$ such that for all $x \in X$, $B(x,\delta) = \set{x}$. We say that a metric space has $C$-bounded geometry for some $C>0$, if there exists a constant $N(C)>0$ such that $|B(x,C)| \leq N(C)$ for all $x \in X$. A metric space has \emph{bounded geometry} if it has $C$-bounded geometry for every $C>0$.
\end{defn}
\begin{exmp}
Note that $C$-bounded geometry for some $C$ does not in general imply bounded geometry. As an easy example, one can consider an infinite metric space equipped with the discrete metric, i.e. $d(x,y)=1$ for every $x,y\in X$ distinct.
\end{exmp} 

\begin{defn} \label{cplp}
  Let $Y$ and $Z$ be metric spaces. A $C$-dense map $p \colon Y \to Z$ has the \emph{coarse path lifting property} if there exists a non-decreasing function $\theta \colon \mathbb{R}^+ \to \mathbb{R}^+$, such that for any $z,z' \in Z$ and $y \in Y$ with $d_Y(p(y), z) \leq C$ there exists a $y' \in Y$ with $d(p(y'),z') \leq C$ and $d(y,y') \leq \theta (d(z,z'))$.
 \end{defn}

  \begin{defn}
  A map between metric spaces $f \colon Y \to Z$ is \emph{bornologous} if for every $R > 0$ there is $S_R > 0$ such that if $d_Y(y,y') \leq R$ then $d_Z(f(y),f(y')) \leq S_R$ for all $y,y' \in Y$.
 \end{defn}

\begin{exmp}
The path lifting property occurs naturally for example in the setting of groups. Let $Y=H$ be a group and let $N\triangleleft H$. The most natural way of defining a distance function on $Z:=H/N$ is by setting $d(hN,h'N)$ to be the infimum of $d(hn,h'n')$ over all $n,n'\in N$. The projection map $p:H\rightarrow H/N$ is a bornologous map and one checks easily that it satisfies the coarse path lifting property. Actually, one only needs the fact that $N$ is ``almost normal'' in $H$, i.e. that for every finite subset $F$ of $H$, there exists a finite subset $F'\subset H$ with $NF\subset F'N$.

Another example can be obtained by taking $Z$ to be the set of right $N$-cosets of $H$, where $N$ is any (not necessarily normal) subgroup of $H$. In this case, the projection map $p:H\rightarrow N\backslash H, g\mapsto Ng$ is a bornologous map that has the coarse path lifting property.
 \end{exmp}
 \begin{thm} \label{CEresult}
  Let $X,Y,Z$ be metric spaces and $p \colon Y \to Z$ be a $C$-dense bornologous map with the coarse path lifting property. Let $\theta \colon \mathbb{R}^+ \to \mathbb{R}^{+}$ be a non-decreasing function satisfying the properties in Definition \ref{cplp}. Assume that $Y$ is uniformly discrete and that $Z$ has $C$-bounded geometry. If $X,Y,Z$ are coarsely embeddable into an $L^1$-space, then so is $X \wr_Z^C Y$.
 \end{thm}
 \begin{rem}
Note that, by Proposition \ref{propembed}, the conclusion of the theorem also implies $L^p$-embeddability of $X \wr_Z Y$ for any $p\geq 1$. On the other hand, it is known that $L^p$ embeds isometrically into $L^1$ for $1\leq p \leq 2$. Hence in the formulation of Theorem \ref{CEresult}, we can just as well replace $L^1$-embeddability by $L^p$-embeddability for $1\leq p \leq 2$. 
\end{rem}
 \begin{proof}[Proof of Theorem \ref{CEresult}]
  By Proposition \ref{propembed}, there exists measured walls structures $(X, \sigma)$, $(Y,\nu)$, $(Z, \mu)$ and functions $\rho_X, \rho_Y, \rho_Z, \eta_X, \eta_Y, \eta_Z \colon \mathbb{R}^+ \to \mathbb{R}^+$ increasing to infinity, such that
 \begin{align}
  \rho_X (d_X(x_1, x_2) ) & \leq d_{\sigma}(x_1, x_2)  \leq \eta_X (d_X(x_1, x_2))  \quad \forall x_1, x_2 \in X \label{CEsigma}\\
  \rho_Y (d_Y(y_1, y_2) ) & \leq d_{\nu}(y_1, y_2)  \leq \eta_Y (d_Y(y_1, y_2)) \quad \forall y_1, y_2 \in Y \label{CEnu}\\
  \rho_Z (d_Z(z_1, z_2) ) & \leq d_{\mu}(z_1, z_2)  \leq \eta_Z (d_Z(z_1, z_2)) \quad \forall z_1, z_2 \in Z. \label{CEmu}
 \end{align}
By Theorem \ref{thm:wallsandembeddings}, there exists a measured walls structure $\widetilde{\mu}$ on $\bigoplus_Z X \times Z$ where for $(\mbf{f},z), (\mbf{g},z') \in \bigoplus_Z X \times Z$
\begin{equation*}
 d_{\wmu} ((\mbf{f}, z), (\mbf{g}, z')) = \mu (\{ A : A \vdash \supp{\mbf{f}^{-1} \mbf{g}} \cup \set{z,z'} \}).
\end{equation*}
We have a projection map $p \colon \bigoplus_Z X \times Y \to \bigoplus_Z X \times Z$ where $(\mbf{f}, y) \mapsto (\mbf{f},p(y))$. Using this we can pullback a measured wall structure on $\bigoplus_Z X \times Y$ where
\begin{equation*}
 d_{p \wmu} ((\mbf{f}, y), (\mbf{g}, y')) = d_{\wmu}((\mbf{f},p(y)),(\mbf{g},p(y')))
\end{equation*}
We define three other wall structures, $\widetilde{\sigma}, \widetilde{\nu}$ and $\widetilde{\omega}$, on $X \wr_Z Y$ where
\begin{align*}
 d_{\widetilde{\sigma}} ((\mbf{f},y), (\mbf{g}, y'))  = & \sum_{z \in Z} d_{\sigma}(\mbf{f}(z),\mbf{g}(z)), \\
 d_{\widetilde{\nu}}((\mbf{f},y), (\mbf{g}, y')) = & \ d_{\nu}(y,y'), \\
 d_{\widetilde{\omega}}((\mbf{f},y), (\mbf{g}, y')) = & \lvert \mbox{Supp}(\mbf{f^{-1}g}) \rvert.
\end{align*}
It is clear from our comments in Section \ref{section:wallspaces} on pushing forward and summing up wall space structures, that $\widetilde{\sigma}$ and $\widetilde{\nu}$ are indeed wall space structures.
To see that the latter is associated to a measure wall space structure, note that $d_{\widetilde{\omega}}$ is associated as in Proposition \ref{propembed} to the map $\Lambda:\bigoplus_Z X \times Y \rightarrow L^1(X\times Z), (\mbf{f},y) \mapsto \Lambda(\mbf{f},y)$, where
\[ \Lambda(\mbf{f},y): (x,z) \mapsto \begin{array}{l} 1/2 \mbox{ if } f(z)=x\\ 0 \mbox{ if } f(z)\neq x. \end{array} \]

We now aim to show that we can coarsely embed $X \wr_Z Y$ into an $L^1$-space. Define $\lambda = p \wmu + \widetilde{\sigma} + \widetilde{\nu}+\widetilde{\omega}$ to be a measured wall space structure on $X \wr_Z Y$. By Proposition \ref{propembed}, it suffices to show that for every $R >0$ if $d_{\lambda}((\mbf{f}, y),(\mbf{g},y')) \leq R$ then $d_{X \wr_Z Y} ((\mbf{f}, y),(\mbf{g},y')) \leq C_1(R)$ and if $d_{X \wr_Z Y}((\mbf{f}, y),(\mbf{g},y')) \leq R$ then $d_{\lambda}((\mbf{f}, y),(\mbf{g},y')) \leq C_2(R)$ where $C_1, C_2$ are constants depending only on $R$.

Fix $R>0$ and suppose $d_{\lambda}((\mbf{f}, y),(\mbf{g},y')) \leq R$. In particular 
\begin{align}
 d_{\wmu}((\mbf{f}, p(y)),(\mbf{g},p(y'))) & \leq R \label{mu} \\
 \sum_{z \in Z} d_{\sigma}((\mbf{f}(z),\mbf{g}(z)) & \leq R  \label{sigma}\\
 d_{\nu}(y,y') & \leq R \label{nu}\\
 \lvert \mbox{Supp}(\mbf{f^{-1}g}) \rvert & \leq R \label{omega}
 \end{align}
  Define $p(y)=z_0$ and write $\supp{\mbf{f}^{-1}\mbf{g}}=\{z_1,z_2,\ldots ,z_n\}$ for some $n\leq R$. By \eqref{mu} it follows that $\mu(A : A \vdash \supp{\mbf{f}^{-1} \mbf{g}} \cup \set{p(y),p(y')}) \leq R$. In particular $d_{\mu}(z_i,z_j) \leq R$ for all $z_i,z_j \in \supp{\mbf{f}^{-1} \mbf{g}} \cup \set{p(y),p(y')}$.  By Equation \eqref{CEmu},  this implies that $d_Z(z_i,z_j)\leq \rho_Z^{-1}(R)$ for all $z_i,z_j$. Starting from $y_0=y$, by the path lifting property, we can find $y_1$ such that $d_Z(p(y_1),z_1)\leq C$ and $d_Y(y,y_1)\leq \theta(\rho_Z^{-1}(R))$. We can then find $y_2$ with $d_Z(p(y_2),z_2)\leq C$ and $d_Y(y_1,y_2)\leq \theta(\rho_Z^{-1}(R))$. Continuing inductively and by the triangle inequality, we obtain
\[ \sum_{i=0}^{n-1} d_Y(y_i,y_{i+1}) +d_Y(y_n,y_0) \leq 2 \sum_{i=0}^{n-1} \theta(\rho_Z^{-1}(R))\leq 2R \theta(\rho_Z^{-1}(R)).\]
Using Equation \eqref{nu} and denoting $y_0=y$, we thus have that 
 \begin{equation}\label{longpath}
\path{I}{y}{y'} \leq  \sum_{i=0}^{n-1} d_Y(y_i,y_{i+1}) +d_Y(y_n,y_0) +d_Y(y,y')\leq 2R \theta(\rho_Z^{-1}(R)) + \rho_Y^{-1}(R)
  \end{equation}
%
%
  %
  Now we can deduce that
  \begin{align*}
   \path{I}{y}{y'} & +  \sum_{z \in Z} d_X((\mbf{f}(z),\mbf{g}(z)) \\ & \leq  2R \theta(\rho_Z^{-1}(R)) + \rho_Y^{-1}(R) 
  + \sum_{z \in \mathrm{Supp}(\mbf{f}^{-1}\mbf{g})} \rho_X^{-1}(R) \qquad \mbox{by \eqref{CEsigma}, \eqref{sigma} and \eqref{longpath}}\\
   & \leq  2R \theta(\rho_Z^{-1}(R)) + \rho_Y^{-1}(R) + R \rho_X^{-1}(R)  \qquad \mbox{by \eqref{omega}}
  \end{align*}
 It suffices to set $C_1(R) =  2R \theta(\rho_Z^{-1}(R)) + \rho_Y^{-1}(R) + R \rho_X^{-1}(R)$.

Now suppose conversely that $d_{X \wr_Z Y}((\mbf{f}, y),(\mbf{g},y')) \leq R$. In particular
  \begin{align}
   \path{I}{y}{y'} & \leq R \label{originaldistance}\\
   \sum_{z \in Z} d_X (f(z), g(z)) & \leq R \label{sumsigma}
  \end{align}
  Let $(y_1, \ldots, y_n) \in \mathcal{P}_I$ such that 
	\begin{equation}
	d_Y(y,y_1) + \sum_{i=1}^{n-1}d_Y(y_i, y_{i+1}) + d_Y(y_n, y') \leq R+1.
	\label{pathY}
	\end{equation}
	As $Y$ is uniformly discrete, we have $\delta_Y:= \inf(d(a,b)\mid a,b\in Y)>0$. This implies that, although some of the $y_i$ may be equal, the number of distinct $y_i$ is bounded by $\frac{R+1}{\delta_Y}$. Any point in the support of $\mbf{f^{-1}g}$ lies, by definition, in a $C$-neighbourhood of some $p(y_i)$. As such neighbourhoods contain at most $N(C)$ elements, we can conclude that
	\begin{equation} \label{eq:support}
	n=\lvert \supp{\mbf{f^{-1}g}} \rvert \leq E(R):=N(C)\frac{R+1}{\delta_Y}.
	\end{equation}
	
	From Equation \eqref{pathY} and the triangle inequality, it follows that
  \begin{equation} \label{ballinY}
d_Y(a,b) \leq R+1 \qquad \forall a,b \in \set{y, y', y_1, \ldots, y_n} 
  \end{equation}
  
 As $p$ is bornologous, there exists $S = S(R+1)$ such that for all $z,z' \in \set{p(y),p(y'),p(y_1), \ldots,p(y_n)}$, we have $d_Z(z,z') \leq S$ . By definition of $(y_1, \ldots, y_n)$ and the triangle inequality it follows that $d_Z(z,z') \leq S+2C$ for every $z,z' \in \supp{\mbf{f}^{-1} \mbf{g}} \cup \set{p(y),p(y')}$. By \eqref{CEmu} it follows that  
  \begin{equation} \label{borno}
d_{\mu}(z,z') \leq \eta_Z(S+2C) \qquad \forall z,z' \in  \supp{\mbf{f}^{-1} \mbf{g}} \cup \set{p(y),p(y')}. 
  \end{equation}
	 Let us enumerate $\supp{\mbf{f}^{-1} \mbf{g}} \sqcup \set{p(y),p(y')}=\{p(y)=z_0,z_1,\ldots ,z_{n+1}=p(y')\}$. Note that, if $A$ cuts $\supp{\mbf{f}^{-1} \mbf{g}} \cup \set{p(y),p(y')}$, then A must cut $\{z_i,z_{i+1}\}$ for some $i\in \{0,1,\ldots ,m-1\}$. Hence $d_{p \wmu}((\mbf{f},y),(\mbf{g},y'))\leq \sum_{i=0}^{n} d_\mu(z_i,z_{i+1})$.
	
	It now follows that
  \begin{align*}
   d_\lambda( (\mbf{f},y), (\mbf{g},y') ) = &   (d_{\widetilde{\nu}} +d_{p \wmu}+ d_{\widetilde{\sigma} } +d_{\widetilde{\omega}} )( (\mbf{f},y), (\mbf{g},y') )\\
	\leq &  \ d_{\nu} (y,y') + \sum_{i=0}^{n} d_{\mu} (z_i,z_{i+1})  + \sum_{z \in Z } d_\sigma (\mbf{f}(z), \mbf{g}(z)) +d_{\widetilde{\omega}} ( (\mbf{f},y), (\mbf{g},y') )\\
   \leq & \ \eta_Y(R) + \sum_{i=0}^{n} \eta_Z(S+2C) + \sum_{z \in Z } \eta_X(R) +E(R) \qquad \mbox{by \eqref{CEnu}, \eqref{sumsigma}, \eqref{borno}, \eqref{eq:support}}\\
   \leq & \ \eta_Y(R) + E(R) \eta_Z(S+2C) + E(R) \eta_X(R) +E(R)
  \end{align*}
  Hence, it suffices to set $C_2(R) :=\eta_Y(R) + E(R)(\eta_Z(S+2C) + \eta_X(R) +1)$. This shows by Proposition \ref{propembed} that $X\wr_Z Y$ embeds coarsely into an $L^p$-space. 
  \end{proof}
 \begin{rem} \label{rem:equivalent}
The only time that we used the conditions 
\begin{enumerate}
\item $Y$ is uniformly discrete,
\item $Z$ has $C$-bounded geometry,
\end{enumerate}
was to show that Equations \eqref{originaldistance} and \eqref{sumsigma} imply that $\lvert \supp{\mbf{f^{-1}g}}\rvert$ is bounded by some function of $R$. One checks easily that the above conditions could equally well be replaced by the following two conditions:
\begin{enumerate}
\item $Y$ has bounded geometry
\item $Z$ has $C$-bounded geometry.
\end{enumerate}
Alternatively, it would also be sufficient to require nothing on $Y$ and $Z$ but to ask that $X$ is a uniformly discrete metric space.
\end{rem}

\section{The compression of $X\wr_Z Y$ in terms of that of $X,Y,Z$}
We can modify the previous proof to give information on the $L^1$-compression of $X\wr_Z Y$ in terms of the growth behaviour of $\theta$ and the $L^1$-compression of $X,Y$ and $Z$.
\begin{defn}
  Let $Y$ and $Z$ be metric spaces and let $p \colon Y \to Z$ be a $C$-dense map with the coarse path lifting property, i.e. there exists a non-decreasing function $\theta \colon \mathbb{R}^+ \to \mathbb{R}^+$, such that for any $z,z' \in Z$ and $y \in Y$ with $d_Y(p(y), z) \leq C$ there exists a $y' \in Y$ with $d(p(y'),z') \leq C$ and $d(y,y') \leq \theta (d(z,z'))$. If $\delta>0$ is such that $\theta(r)\lesssim r^\delta+1$ for every $r\in \mathbb{R}^+$, then we say that $p$ has the \emph{$\delta$-polynomial path lifting property}. We say that $p$ has the polynomial path lifting property if it has the \emph{$\delta$-polynomial path lifting property} for some $\delta>0$.
\end{defn}
  \begin{thm}
	 Let $X,Y,Z$ be metric spaces and $p \colon Y \to Z$ be a $C$-dense bornologous map with the coarse path lifting property. Let $\theta \colon \mathbb{R}^+ \to \mathbb{R}$ be a non-decreasing function satisfying the properties in Definition \ref{cplp}. Assume that either ``$X$ is uniformly discrete'' or ``$Y$ is uniformly discrete and $Z$ has $C$-bounded geometry''. \\
	
Assume that there are constants $a,b>0$ such that $d_Z(p(y),p(y'))\leq ad_Y(y,y')+b$ for every $y,y'\in Y$. If $p$ has the $\delta$-polynomial path lifting property for some $\delta>0$ and if $X,Y,Z$ have $L^1$-compression equal to $\alpha,\beta,\gamma$ respectively, then the $L^1$-compression of $X\wr_Z^C Y$ is bounded from below by $\min(\alpha,\beta,\frac{\gamma}{\gamma+\delta})$. \label{thm:comp}
  \end{thm}
	\begin{rem}
 Our bound generalizes the bound of Theorem $1.1$ in \cite{Li2010}. Note further that, as both $X$ and $Y$ can be considered as metric subspaces of $X\wr_Z Y$, one also has an upper bound, namely $\min(\alpha,\beta)$, for the compression of $X\wr_Z Y$.
	\end{rem}
	\begin{proof}[Proof of Theorem \ref{thm:comp}]
The starting point for this proof is the proof of Theorem \ref{CEresult} and we will often refer to inequalities stated there. For now, assume that $\alpha,\beta,\gamma$ are real numbers and that $f_1 \colon X \to L^1$, $f_2 \colon X \to L^1$ and $f_3 \colon Z \to L^1$ are large scale Lipschitz functions into $L^1$-spaces such that
\begin{align*}
 d_X(x,x')^{\alpha} \lesssim \norm{f_1(x) - f_1(x')}_1 \\
 d_Y(y,y')^{\beta}\lesssim \norm{f_2(y) - f_2 (y')}_1 \\
 d_{Z}(z,z')^{\gamma} \lesssim \norm{f_3(z) - f_3(z')}_1. \\
\end{align*}
Here, $\lesssim$ denotes inequality up to a multiplicative constant. The reason that we can take the lower bounds as above, is that by taking the direct sum of $f_i$ with the coarse embedding $\widetilde{f_i}:W\rightarrow l^2(W), w\mapsto \delta_w$, where $W=X,Y$ or $Z$, we can always assume that $\norm{f_1(w) - f_1(w')}_1\geq 1$ for distinct $w,w'$.

Let $d_{\sigma}$, $d_{\nu}$, and $d_{\mu}$ be the measured wall space structures associated to the functions $f_1, f_2, f_3$ by Proposition \ref{propembed}. Define the measured walls $d_{p\wmu}, d_{\widetilde \mu}, d_{\widetilde \sigma},d_{\widetilde{\omega}}$ on $X \wr_{Z} Y$ as in theorem \ref{CEresult}. As a first step, we are going to show that the function associated to the measured wall $d_\lambda=d_{p\wmu} + d_{\widetilde \mu} + d_{\widetilde \sigma}+d_{\widetilde{\omega}}$ is Lipschitz. That is, there is a constant $\widetilde{C}\in \mathbb{R}$ such that for every $(\mbf{f}, y), (\mbf{g}, y') \in X \wr_{Z} Y$,
\begin{equation*}
d_{p \widetilde{\mu}}((\mbf{f},y),(\mbf{g},y')) + d_{\widetilde \nu}((\mbf{f},y),(\mbf{g},y')) + d_{\widehat{\sigma}}((\mbf{f},y),(\mbf{g},y'))+d_{\widetilde{\omega}}(((\mbf{f},y),(\mbf{g},y')) \leq \widetilde{C} d_{X\wr_Z Y}((\mbf{f},y),(\mbf{g},y')).
\end{equation*}
By Equation \eqref{eq:support}, it follows that $d_{\widetilde{\omega}}$ corresponds to a large-scale Lipschitz function if $Y$ is uniformly discrete and $Z$ has $C$-bounded geometry. Starting from Equation \eqref{originaldistance} and \eqref{sumsigma}, one can easily show the same fact using only uniform discreteness of $X$.

As $d_\nu$ and $d_{\sigma}$ both correspond to large scale Lipschitz functions, this implies that so does $d_{\widetilde \nu} + d_{\widetilde \sigma}$:
\begin{align*}
 d_{\widetilde\nu}((\mbf{f},y),(\mbf{g},y')) + d_{\widetilde{\sigma}}((\mbf{f},y),(\mbf{g},y')) = & \ d_{\nu}(y,y') + \sum_{z \in Z} d_{\sigma}(\mbf{f}(z), \mbf{g}(z))\\
 \lesssim & \ d_Y(y,y') + 1+ \sum_{z \in Z} d_X(\mbf{f}(z), \mbf{g}(z)) + d_{\widetilde{\omega}}((\mbf{f},y),(\mbf{g},y')) \\
\lesssim & d_{X\wr_Z Y}((\mbf{f},y), (\mbf{g},y'))+1.
\end{align*}
It thus remains to show that $d_{p \widetilde{\mu}}$ corresponds to a Lipschitz function. 
Denote $y_0=y, y_{n+1}=y'$ and choose $(y_1, \ldots, y_n) \in \mathcal{P}_I$ such that 
\begin{equation*}
 \path{I}{y}{y'} \leq \sum_{i=0}^{n} d_Y(y_i, y_{i+1}) \leq \path{I}{y}{y'} + 1
\end{equation*}
Write $z_0=p(y),\ z_{n+1}=p(y')$ and enumerate the elements of $\supp{\mbf{f}^{-1} \mbf{g}}$ as $\{z_1,z_2,\ldots ,z_n\}$ where each $z_i$ lies in a $C$-ball around $p(y_i)$. As $p$ is bornologous, we have that $d_Z(z_i,z_{i+1})\leq 2C+ad(y_i,y_{i+1})+b$ for each $i$. Hence,
\begin{align*}
d_{p \widetilde{\mu}}((\mbf{f},y),(\mbf{g},y'))& \leq  \sum_{i=0}^{n} d_\mu(z_i,z_{i+1})\\
& \lesssim   \sum_{i=0}^{n} d_Z(z_i,z_{i+1}) + d_{\widetilde{\omega}}((\mbf{f},y),(\mbf{g},y'))\\
& \leq  n(2C+b)+a\sum_{i=0}^{n} d_Y(y_i, y_{i+1})+d_{\widetilde{\omega}}((\mbf{f},y),(\mbf{g},y'))\\
& =  d_{\widetilde{\omega}}((\mbf{f},y),(\mbf{g},y')) (2C+b+1)+a\sum_{i=0}^{n} d_Y(y_i, y_{i+1})\\
&\leq  d_{\widetilde{\omega}}((\mbf{f},y),(\mbf{g},y')) (2C+b+1) + a + a\ \path{I}{y}{y'} \\
&\lesssim  d_{X\wr_Z Y}((\mbf{f},y),(\mbf{g},y')) +1,
\end{align*}
where we use that $d_{\widetilde{\omega}}$ corresponds to a large-scale Lipschitz function.
We conclude that $d_{\lambda}$ is associated to a large scale Lipschitz map of $X\wr_Z Y$ into an $L^1$-space.

As a second step, we calculate the compression of $d_\lambda$. Assume first that $d_\lambda((\mbf{f},y),(\mbf{g},y'))\leq R$ for some $R>0$ such that Equations \eqref{mu}, \eqref{sigma}, \eqref{nu} and \eqref{omega} are valid. Enumerate the elements of $\supp{\mbf{f}^{-1}\mbf{g}}$, say $z_1, z_2, \ldots, z_n$. Set $z_0=p(y)$. Denote $y_0=y$, then use the path lifting property to take $y_1$ such that $d_Z(p(y_1),z_1)<C$ and $d(y_0,y_1)\leq ad(z_0,z_1)^\delta+b$. Next, take $y_2$ such that $d_Z(p(y_2),z_2)<C$ and such that $d(y_1,y_2)\leq a d_Z(z_1,z_2)^\delta+b$ and so on. By definition, we have
\[ \path{I}{y}{y'}\leq (\sum_{i=0}^{n-1} d_Y(y_i,y_{i+1})) + d_Y(y_n,y').\]
We now obtain
\begin{align*}
\path{I}{y}{y'} & \leq  \sum_{i=0}^{n-1}d_Y(y_i,y_{i+1}) +d_Y(y_n,y') \\
&\lesssim  \sum_{i=0}^{n-1}d_Y(y_i,y_{i+1}) +d_Y(y,y')\\
&\lesssim  \sum_{i=0}^{n-1} (d_Z(z_i,z_{i+1})^\delta +1) +d_Y(y,y')\\
&\lesssim  R+ \sum_{i=0}^{n-1} d_Z(z_i,z_{i+1})^\delta + d_\nu(y,y')^{1/\beta}\\
&\leq  R+ \sum_{i=0}^{n-1} d_Z(z_i,z_{i+1})^\delta +R^{1/\beta}\\
&\lesssim  R+ \sum_{i=0}^{n-1} d_\mu(z_i,z_{i+1})^{\delta/\gamma} +R^{1/\beta}\\
& \lesssim R+ RR^{\delta/\gamma} + R^{1/\beta}, 
\end{align*}
where the last inequality follows from the fact that
\[ d_\mu(z_i,z_{i+1})\leq d_{p\widetilde{\mu}}((\mbf{f},y),(\mbf{g},y'))\leq R.\]
Consequently, we obtain
\begin{align*}
d_{X\wr_Z Y}((\mbf{f},y),(\mbf{g},y')) & =  \path{I}{y}{y'} + \sum_{z\in Z} d_X(f(z),g(z)) \\
& \lesssim  R^{\frac{\delta}{\gamma} +1} + R^{1/\beta}+ \sum_{z \in Z} d_\sigma(f(z),g(z))^{1/\alpha} \\
& \lesssim  R^{\frac{\delta+\gamma}{\gamma}} + R^{1/\beta}+ (\sum_{z \in Z} d_\sigma(f(z),g(z)))^{1/\alpha} \\
& \lesssim  R^X, 
\end{align*}
where $X=\max(\frac{\delta+\gamma}{\gamma}, \frac{1}{\alpha},\frac{1}{\beta})$.
Consequently, the compression of $d_\lambda$, and hence of $X\wr_Z Y$ is bounded from below by
\[ \min(\alpha, \beta, \frac{\gamma}{\delta+\gamma}). \qedhere \] 
\end{proof}

\begin{rem}
At the end of Section 2 in \cite{Li2010}, the author shows that the $L^p$-compression $\alpha_p^*(X)$ of a metric space $X$ is always greater than $\max(\frac{1}{2},\frac{1}{p}) \alpha_1^*(X)$. Moreover, $L^p$ embeds isometrically into $L^1$ for any $p\in [1,2]$. So, for $p\in [1,2]$, we deduce that the positivity of the $L^p$-compression is preserved under generalized wreath products with the polynomial path lifting property.
\end{rem}
%

\section{Box spaces of wreath products}
Let $\{K_i\}$ be some collection of nested finite index normal subgroups of a finitely generated residually finite group $G$, for which the intersection $\cap_{n\in \mathbb{N}} K_n$ is trivial.
\begin{defn}
The \emph{box space} of $G$ corresponding to $\{K_i\}$, denoted by $\Box_{\{K_i\}} G$, is the disjoint union $\sqcup_i G/K_i$ of finite quotient groups of $G$, where each quotient is endowed with the metric induced by the image of the generating set of $G$, and the distance between any two distinct quotients is chosen to be greater than the maximum of their diameters. 
\label{def:boxspace}
\end{defn}
The box space of a group is thus only defined when the group is residually finite. The following theorem of Gruenberg tells us when this happens for a wreath product of two groups.

\begin{thm}[\cite{Gru}, Theorem 3.2]
Let $G$ and $H$ be residually finite groups. Then the wreath product $G\wr H$ is residually finite if and only if either $H$ is finite or $G$ is abelian.
\end{thm}

We will focus on the interesting case when $H$ is infinite and $G$ is abelian. We will use the following lemma from \cite{Gru}, which forms part of the proof of the theorem above. 

\begin{lem}[\cite{Gru}, Lemma 3.2]
If $G$ is an abelian group, then any surjective homomorphism from a group $H$ to a group $K$ can be extended in a natural way to a surjective homomorphism from the wreath product $G\wr H$ to $G\wr K$.
\end{lem}

We note that it follows from the proof that the kernel of the resulting surjective homomorphism is the smallest normal subgroup of $G\wr H$ containing the kernel of the surjective homomorphism from $H$ to $K$, i.e. the normal closure of this kernel in $G\wr H$.

\begin{thm}\label{BoxResult}
Let $G$ be a finitely generated abelian group and let $H$ be a finitely generated residually finite group which has a box space which embeds coarsely into a Hilbert space. Then there is a box space of the wreath product $G\wr H$ which coarsely embeds into a Hilbert space. 
\end{thm}

\begin{proof}
By Theorem 3.2 of Gruenberg (\cite{Gru}), it makes sense to talk about box spaces of $G\wr H$ since it is residually finite. 

Using the above-mentioned Lemma 3.2 of \cite{Gru}, given any normal subgroup $N$ of $H$, we have a surjective homomorphism from $G\wr H$ onto $G\wr (H/N)$. 
Now given a normal subgroup $K$ of one of the direct summands $G$ of $\bigoplus_{H/N} G$, its normal closure in $G\wr (H/N)$ will be $\bigoplus_{H/N} K$. One can then take the quotient of $G\wr (H/N)$ by the normal closure of $K$ to get $$\bigoplus_{H/N} G/K \rtimes H/N \cong (G/K)\wr(H/N).$$

We will now build the embeddable box space of $G\wr H$. 
Let $\{N_i\}$ be a nested sequence of normal finite index subgroups of $H$ such that the corresponding box space $\Box_{\{N_i\}} H$ embeds coarsely into a Hilbert space. Since $G$ is finitely generated abelian, and hence residually finite and amenable, Guentner's result tells us that the box space $\Box_{\{K_i\}} G$ of $G$ with respect to any collection of nested finite index normal subgroups $\{K_i\}$ with trivial intersection has property $A$ and thus coarsely embeds into a Hilbert space. 

For each $i\in \mathbb{N}$, taking the quotient as described above yields $(G/K_i)\wr (H/N_i)$ as a quotient of $G\wr H$. Note that this quotient is finite. The kernel, which we will denote by $M_i$, is thus of finite index in $G\wr H$. 

To show that $\Box_{\{M_i\}} G\wr H$ is indeed a box space, it remains to check that the sequence of normal subgroups $\{M_i\}$ is nested and has trivial intersection. 

It is easy to see that $\cap_i M_i= \{1\}$, since this is equivalent to showing that for each non-trivial element of $G\wr H$, there is some $i$ such that the image of this element remains non-trivial in the finite quotient $(G\wr H)/M_i$. So, given a non-trivial element $(\gamma, h) \in G\wr H$, observe that if $h\neq 1$, then there is some $i$ for which the image of $h$ in the quotient $H/N_i$ is non-trivial. Thus, $(\gamma, h)$ remains non-trivial in $(G\wr H)/M_i$, so this is the required $i$. 

If $h$ is trivial, then $\gamma\in \bigoplus_{H} G$ must have at least one non-trivial entry $g\in G$ in some direct summand. Take $j$ large enough so that the images of the elements of $H$, for which $\gamma$ has a non-trivial entry in the corresponding direct summand, are distinct and non-trivial in $H/N_j$. This can be done since $H$ is residually finite.
Since the $K_i$ have trivial intersection, there is some index $i>j$ such that $g$ has non-trivial image in the quotient $G/K_i$. Hence the image of $(\gamma, 1)$ must be non-trivial in the quotient $(G\wr H)/M_i$. 

To show that the $M_i$ are nested we just need to observe that each $M_i$ is the normal closure of the subgroup generated by $N_i$ and $K_i$ in $G\wr H$. This is easy to see, having noted that the kernel of the surjection from $G\wr H$ to $G\wr H/N_i$ is the normal closure of $N_i$ in $G\wr H$. Since the $K_i$ and $N_i$ are nested, so are the normal closures of the subgroups generated by them. 

Let $\psi:\Box_{\{N_i\}}H \longrightarrow \ell^1$ and $\phi: \Box_{\{K_i\}}G \longrightarrow \ell^1$ be coarse embeddings into $\ell^1$ (we know these exist since $\ell^2$ coarsely embeds into $\ell^1$), and let $\rho_{\pm}$ and $\tau_{\pm}$ be the strictly increasing unbounded compression functions for the embeddings $\psi$ and $\phi$ respectively. 

For each pair of quotients $H/N_i$ and $G/K_i$ in the box spaces, our Theorem \ref{CEresult} implies that the wreath product $(G/K_i)\wr(H/N_i)\cong (G\wr H)/M_i$ coarsely embeds into a Hilbert space with compression functions $\nu_{\pm}$ which depend only on $\rho_{\pm}$ and $\tau_{\pm}$. Actually, instead of using our Theorem \ref{CEresult}, which generalizes Li's Theorem $1.2$ in \cite{Li2010}, we could in this setting just as well use Li's original result for $p=2$.

So, the disjoint union $\sqcup_i(G/K_i)\wr(H/N_i)$ (that is, the box space $\Box_{\{M_i\}}G\wr H$) coarsely embeds into a Hilbert space, with compression functions depending only on $\nu_{\pm}$ and the chosen distances between quotients. Thus, we have found an embeddable box space of $G\wr H$. 
\end{proof}

\begin{rem}
If $H$ is finite, given an embeddable box space $\Box_{\{K_i\}}G$ of $G$, it is easy to construct an embeddable box space of $G \wr H$. Let $N_i$ be the normal closure in $G\wr H$ of the subgroup $K_i$ of one of the summands $G$ of $\bigoplus_H G$. It is clear that the resulting subgroups $\{N_i\}$ form a nested sequence of finite index normal subgroups of $G\wr H$ with trivial intersection, and that the box space $\Box_{\{N_i\}} G\wr H$ coarsely embeds into a Hilbert space. 
\end{rem}

 \section*{Acknowledgement}
The first and second author thank Jacek Brodzki for interesting discussions. The third author wishes to thank Yves Stalder for stimulating conversations.


\begin{thebibliography}{10}

\bibitem{Arzhantseva2012}
Goulnara Arzhantseva, Erik Guentner, and J{\'a}n {\v{S}}pakula.
\newblock Coarse non-amenability and coarse embeddings.
\newblock {\em Geom. Funct. Anal.}, 22(1):22--36, 2012.

\bibitem{Chatterji2010}
Indira Chatterji, Cornelia Dru{\c{t}}u, and Fr{\'e}d{\'e}ric Haglund.
\newblock Kazhdan and {H}aagerup properties from the median viewpoint.
\newblock {\em Adv. Math.}, 225(2):882--921, 2010.

\bibitem{CCJJV01}
Pierre-Alain Cherix, Michael Cowling, Paul Jolissaint, Pierre Julg, and Alain
  Valette.
\newblock {\em Groups with the {H}aagerup property}, volume 197 of {\em
  Progress in Mathematics}.
\newblock Birkh\"auser Verlag, Basel, 2001.
\newblock Gromov's a-T-menability.

\bibitem{Ioana}
Ionut Chifan and Adrian Ioana.
\newblock On relative property ({T}) and {H}aagerup's property.
\newblock {\em Trans. Amer. Math. Soc.}, 363(12):6407--6420, 2011.

\bibitem{CSV12}
Yves Cornulier, Yves Stalder, and Alain Valette.
\newblock Proper actions of wreath products and generalizations.
\newblock {\em Trans. Amer. Math. Soc.}, 364(6):3159--3184, 2012.

\bibitem{DG08}
Marius Dadarlat and Erik Guentner.
\newblock Uniform embeddability of relatively hyperbolic groups.
\newblock {\em J. Reine Angew. Math.}, 612:1--15, 2007.

\bibitem{dCTV07}
Yves de~Cornulier, Romain Tessera, and Alain Valette.
\newblock Isometric group actions on {H}ilbert spaces: growth of cocycles.
\newblock {\em Geom. Funct. Anal.}, 17(3):770--792, 2007.

\bibitem{Erschler2006}
Anna Erschler.
\newblock Generalized wreath products.
\newblock {\em Int. Math. Res. Not.}, pages Art. ID 57835, 14, 2006.

\bibitem{Ferry1995}
Steven~C. Ferry, Andrew Ranicki, and Jonathan Rosenberg, editors.
\newblock {\em Novikov conjectures, index theorems and rigidity. {V}ol. 1},
  volume 226 of {\em London Mathematical Society Lecture Note Series}.
\newblock Cambridge University Press, Cambridge, 1995.
\newblock Including papers from the conference held at the Mathematisches
  Forschungsinstitut Oberwolfach, Oberwolfach, September 6--10, 1993.

\bibitem{G01}
M.~Gromov.
\newblock Asymptotic invariants of infinite groups.
\newblock In {\em Geometric group theory, {V}ol.\ 2 ({S}ussex, 1991)}, volume
  182 of {\em London Math. Soc. Lecture Note Ser.}, pages 1--295. Cambridge
  Univ. Press, Cambridge, 1993.

\bibitem{Gru}
K.~W. Gruenberg.
\newblock Residual properties of infinite soluble groups.
\newblock {\em Proc. London Math. Soc. (3)}, 7:29--62, 1957.

\bibitem{GK04}
Erik Guentner and Jerome Kaminker.
\newblock Exactness and uniform embeddability of discrete groups.
\newblock {\em J. London Math. Soc. (2)}, 70(3):703--718, 2004.

\bibitem{KY}
Gennadi Kasparov and Guoliang Yu.
\newblock The {N}ovikov conjecture and geometry of {B}anach spaces.
\newblock {\em Geom. Topol.}, 16(3):1859--1880, 2012.

\bibitem{Khu}
A.~Khukhro.
\newblock Box spaces, group extensions and coarse embeddings into {H}ilbert
  space.
\newblock {\em J. Funct. Anal.}, 263(1):115--128, 2012.

\bibitem{Li2010}
Sean Li.
\newblock Compression bounds for wreath products.
\newblock {\em Proc. Amer. Math. Soc.}, 138(8):2701--2714, 2010.

\bibitem{R03}
John Roe.
\newblock {\em Lectures on coarse geometry}, volume~31 of {\em University
  Lecture Series}.
\newblock American Mathematical Society, Providence, RI, 2003.

\bibitem{STY}
G.~Skandalis, J.~L. Tu, and G.~Yu.
\newblock The coarse {B}aum-{C}onnes conjecture and groupoids.
\newblock {\em Topology}, 41(4):807--834, 2002.

\bibitem{Y00}
Guoliang Yu.
\newblock The coarse {B}aum--{C}onnes conjecture for spaces which admit a
  uniform embedding into {H}ilbert space.
\newblock {\em Invent. Math.}, 139(1):201--240, 2000.

\end{thebibliography}
\end{document}